\documentclass[a4paper]{amsart}
\usepackage{amsmath, amsfonts,amsthm,amssymb,amscd, verbatim,graphicx,color,multirow,booktabs, caption, bm,chngcntr,adjustbox,longtable,mathtools,microtype, xfrac}
\usepackage{aliascnt}
\usepackage[top = 78pt, bottom = 72pt, inner=108pt, outer=108pt]{geometry}
\usepackage{hyperref}
\usepackage[dvipsnames, table]{xcolor}
\definecolor{LinkBlueSlate}{HTML}{2F4D73}
\hypersetup{colorlinks=true, linkcolor=Maroon, citecolor=OliveGreen, urlcolor=LinkBlueSlate}
\usepackage{tikz, tikz-cd}

\usepackage[capitalise,noabbrev]{cleveref}
\usepackage[shortlabels]{enumitem}

\newcommand{\len}{\mathrm{length}}
\newcommand{\Aut}{\mathrm{Aut}}
\newcommand{\B}{\mathbf{B}}
\newcommand{\Z}{\mathbb{Z}}
\newcommand{\F}{\mathbb{F}}
\newcommand{\M}{\mathcal{M}}

\newtheorem{thm}{Theorem}[]
\newtheorem{lemma}[thm]{Lemma}

\newtheorem{thmx}{Theorem}[]

\newtheorem{corx}[thmx]{Corollary}

\theoremstyle{definition}

\newtheorem{remark}[thm]{Remark}
\newtheorem{construction}[thm]{Construction}
\newtheorem{question}[thmx]{Question}

\author[M.~Barbieri]{Marco Barbieri}
\address{Fakulteta za Matematiko in Fiziko, Univerza v Ljubljani, Jadranska ulica 21, 1000 Ljubljana, Slovenia} 
\email{marco.barbieri@fmf.uni-lj.si}

\author[A.~Zozaya]{Andoni Zozaya} 
\address{Department of Statistics, Computer Science and Mathematics, Public University of Navarra (UPNA) \& Institute for Advanced Materials and Mathematics (INAMAT$^2$), \linebreak Arrosadia Campus, 31006 Pamplona, Spain}
\email{andoni.zozaya@unavarra.es}


\title{Flexible $3$-valent graphs of arbitrary girth}

\thanks{\textit{Funding:} The first author is supported by the Slovenian Research Agency grant J1-50001 and research programme P1-0222. He is also a member of the GNSAGA INdAM research group and gratefully acknowledges its support. The second author is supported by the Spanish Government, grant PID2020-117281GB-I00 (partly funded with ERDF), and by the Public University of Navarra, project 244 \textit{Álgebra. Aplicaciones}. The work on this paper began while the first author was visiting Prof.\ Primo\v{z} Poto\v{c}nik at IMFM in Ljubljana, and the second author was employed at the University of Ljubljana.}

\begin{document}	
	
    \begin{abstract}
        We prove the existence of a finite connected flexible $3$-valent vertex-transitive graph of girth $\ell$ for every integer $\ell\ge 3$.
    \end{abstract}
	
    \maketitle
    \tableofcontents

\section{Introduction}\label{sec:intro}

Vertex-transitive $3$-valent graphs fall into three classes according to the number of edge orbits of their automorphism group. Those with one edge orbit are \emph{arc-transitive}, those with three are \emph{graphical regular representations} of their automorphism group, and the last family has two edge orbits. While the former two families have been extensively studied, the two-edge-orbit case is usually the harshest obstacle in trying to obtain general results about vertex-transitive $3$-valent graphs. (For some recent examples, we suggest \cite{BarbieriGrazianSpiga2025,MaSc,PS21,PSV15,PTV24,Spiga2014}). Following~\cite{Ademir}, a connected vertex-transitive $3$-valent graph is \emph{flexible} if its automorphism group has exactly two edge orbits and every vertex stabilizer has order at least $4$, or, equivalently, each vertex-stabilizer is a noncyclic $2$-group.\footnote{The difficulty in studying \(3\)-valent graphs with two edge orbits lies in the fact that the order of a vertex stabilizer can be arbitrarily large. In contrast, when the stabilizers are cyclic of order \(2\), their behaviour can be completely determined by their action on the neighbourhood of the fixed vertex. This is the reason for the exclusion in the definition of flexible.}

A well-studied parameter in graph theory is the \emph{distinguishing number} of a graph, defined as the minimum number of colours that must be assigned to the vertices of the graph so that no nontrivial automorphism preserves the colouration (see~\cite{AlbertsonCollin}). With only finitely many exceptions, two colours suffice for all $3$-valent vertex-transitive graphs ~\cite{HIKST}. This raises a natural question: \emph{what is the minimum number of vertices which need to be coloured to break all symmetries?} The parameter capturing this idea is called the \emph{$2$-distinguishing cost} of the graph, and it has been introduced in~\cite{Boutin}. It equals $1$ when the automorphism group acts regularly, while $5$ suffices for every finite arc-transitive cubic graph admitting a distinguishing $2$-colouring. In the flexible case, several results relate distinguishing cost to girth~\cite{BarbieriSabatiniSpiga2027,Ademir,ILTW}. This motivated the following~\cite[Question~3.6]{Ademir}: \emph{does there exist a connected flexible $3$-valent vertex-transitive graph of girth $\ell$ for every integer $\ell\ge 3$?} We answer this question in the affirmative.

\begin{thmx}\label[thmx]{thm}
    For every integer $\ell\geq3$, there exists a finite connected flexible $3$-valent vertex-transitive graph $\Gamma_\ell$ of girth $\ell$.
\end{thmx}

The analogous existence problems in the Cayley and arc-transitive settings have been addressed in~\cite{JS11} and in~\cite{arc-transitive}, respectively. In contrast, examples of such flexible graphs are elusive: $14$ was the largest girth for which a flexible graph was known. In fact, all flexible graphs of girth at most $7$ are known, as part of the wider classification of all $3$-valent vertex-transitive graphs of small girth, obtained in~\cite{EibenJajcaySparl, LekseToledo, VidaliPotocnik}. Further examples of girth up to $14$ can be found in the census of $3$-valent vertex-transitive graphs built in~\cite{PotocnikSpigaVerret2013}. (To find them, we used the \textsf{GAP} package \texttt{GraphSym}~\cite{Rhys}.\footnote{For girth $7$, the smallest flexible graph has $98$ vertices and database number $13$. For girth $8$, there are two examples on $64$ vertices, with database numbers $4$ and $6$. For girth $9$, there are three graphs on $486$ vertices, whose database numbers are $56$, $59$, and $71$. For girth $10$, there is a unique graph on $144$ vertices, with database number $68$. For girth $12$, the smallest graph of such girth has $324$ vertices and database number $149$. A single example of girth $14$ has been found, with $864$ vertices and database number $ 1834$. Finally, no flexible graphs of girth $11$, $13$, $15$ or higher having less than $1280$ vertices exist.})

Our proof of \cref{thm} combines three ideas. First, we use the splitting operation to reduce the problem to constructing a suitable $4$-valent graph (see \Cref{sec:split}). This allows us to apply Bass-Serre theory and to work with explicit amalgamated products (see \Cref{sec:BassSerre}). Second, we construct an intermediate infinite quotient of the $4$-valent tree, consisting of cycles of length $\ell$ glued together in a tree structure. This geometry rules out shorter cycles and works uniformly for every girth (see the proof of \cref{lem: res fin}). Finally, we use residual finiteness of a quotient of the amalgamated product to obtain a finite graph while preserving all the required properties (see \Cref{sec:finiteIndex}).

At first, we establish existence without obtaining an explicit bound on the number of vertices of $\Gamma_\ell$. In order to estimate this quantity, we adapt the proof to construct an explicit finite-index normal subgroup of a free group of appropriate girth.\footnote{Here the girth of a normal subgroup $N\unlhd F$ is the girth of the Cayley graph of $F/N$ with respect to the images of the canonical generators of the free group $F$.} (See \Cref{sec: effective}. The study of the girth of subgroups is an area of independent interest -- see, for instance, \cite{BradfordThom2019,BradfordThom2024,Matucci,MaPu2, Rivin}.)

\begin{thmx}\label{cor: p-power-nilpotent}
    Let $F_k = \langle x_1, \dots, x_k \rangle$ be the free group of rank $k$, let $p$ be a prime number, let $\gamma_p(F_k)$ be the $p$-th term of the lower central series of $F_k$, and let $F_k^p$ be the subgroup generated by $p$-powers. Then, for every nonidentity $w \in F_k^p \cdot \gamma_p(F_k)$, $\len(w) \geq p$, and equality is attained if, and only if, $w = x_i^{\pm p}$ for some index $i \in \{1, \dots, k\}$.
\end{thmx}

Recall that, for every integer $a,b \geq 1$, the subgroup $F_k^a \cdot \gamma_b(F_k)$ has finite index in $F_k$: this is precisely Hirsch's solution to the Burnside problem for nilpotent groups (see, for instance,~\cite[Lemma~2.34]{CeSi}). We denote by $\mathsf{N}(k,a,b)$ the order of the largest $k$-generated nilpotent group of class at most $b-1$ and exponent dividing $a$. That is,
\[ \mathsf{N}(k, \, a, \, b) = \left|F_k : F_k^a \gamma_{b}(F_k) \right|  \, .\]
We also recall that, using Hall's basis theorem and the Witt formula,
\[ |\mathsf{N}(k,a,b)| \le a^{\sum_{j=1}^{b-1}k^j} \,. \]
(See, for instance, \cite[Section~2.1]{Bridson}.)

Applying \cref{cor: p-power-nilpotent} to a free subgroup of finite index in the amalgam used for \cref{thm} gives an explicit bound. (Observe also that \cref{cor: p-power-nilpotent} implies the infinitude of flexible graphs of girth $\ell$. See \cref{rem:infinitude}.)

\begin{corx}\label[corx]{prop:bounds}
For every integer $\ell\geq3$, and for every prime $p>2\ell+2$, there exists a finite connected flexible $3$-valent vertex-transitive graph $\Gamma_\ell^p$ of girth $\ell$ with
\[ |V\Gamma_\ell^p| \le \frac{(8 \ell)!}{4}\, \mathsf{N}\left((8 \ell)!, \, p, \, p \right) \,.\]
\end{corx}

The bound from \cref{prop:bounds} is disproportionately high compared to the few examples we found in the census. Hence, we conclude by proposing the following extremal problem.
\begin{question}
What is the smallest number of vertices of a flexible cubic vertex-transitive graph of girth $\ell$?
\end{question}

\subsection{Acknowledgments}
We would like to thank Prof.~Ademir Hujdurović for posing the question to us of whether there exist finite connected vertex-transitive $3$-valent and flexible graphs of arbitrary girth, for suggesting that the operation of splitting could aid in addressing this problem, and for sharing with us the preliminary version of~\cite{Ademir}. We would also like to thank Prof.~Montserrat Casals-Ruiz for helpful conversations, Dr.~Rhys J. Evans for finding the small examples of girth up to $14$, and an anonymous referee for pointing out \Cref{remark:sonoUnaBellissimaFarfalla}.

\subsection{Disclosure of AI usage}
ChatGPT 5.6 has been used to proof-read the note. Moreover, the backtracking argument in the proof of \cref{lemma:finiteIndexNormal} was suggested by ChatGPT 5.6 to the first author while working on an unrelated problem.

\section{The operation of splitting}\label{sec:split}
    The operations of \emph{splitting} and \emph{merging} have been introduced in~\cite[Constructions~7 and~11]{PotocnikSpigaVerret2013} to create a framework for translating results from $3$-valent graphs with a perfect matching invariant under the action of their automorphism groups into results for $4$-valent graphs with a $2$-factor decomposition invariant under the action of their automorphism groups, and \textit{vice versa}. The fact that these operations are the inverse to one another (on both sides) has been proved in~\cite[Theorem~12]{PotocnikSpigaVerret2013} and~\cite[Theorem~2.9]{BarbieriGrazianSpiga2025}, and it suggests that we can reduce the proof of \cref{thm} to a problem on $4$-valent graphs with a $G$-invariant cycle decomposition. For our purposes, we only need to give a precise definition of the operation of splitting.

    \begin{construction}\label[construction]{con:splitting}
        Let $\Delta$ be a $4$-valent graph, and let $\mathcal{C}$ be a partition of $E\Delta$ into cycles. We build a $3$-valent graph, $\mathrm{s}(\Delta,\mathcal{C})$, whose vertex-set is
        \[ V\mathrm{s}(\Delta,\mathcal{C}) = \left\{(\alpha,\mathbf{c}) \in V\Delta \times \mathcal{C} \mid \alpha \in V\mathbf{c} \right\} \,,\]
        and such that two vertices $(\alpha,\mathbf{c})$ and $(\beta,\mathbf{d})$ are declared adjacent if either $\mathbf{c}$ and $\mathbf{d}$ are distinct and $\alpha= \beta$, or $\mathbf{c}=\mathbf{d}$ and $\alpha$ and $\beta$ are adjacent in $\mathbf{c}=\mathbf{d}$. The graph $\mathrm{s}(\Delta,\mathcal{C})$ is the \emph{split graph of the pair $(\Delta,\mathcal{C})$}.
    \end{construction}

    We need to point out some properties of \cref{con:splitting} that will be used below.
    
    Observe that, since $\Delta$ is $4$-valent, there are precisely two cycles in $\mathcal{C}$ passing through $\alpha$. Consequently, $\mathrm{s}(\Delta,\mathcal{C})$ is $3$-valent and
    \[ |V\mathrm{s}(\Delta,\mathcal{C})| = 2 |V\Delta| \,.\]
    
    Let $G\le \Aut(\Delta)$ be a group of automorphisms of $\Delta$. Note that, if $G$ fixes the edge-partition $\mathcal{C}$ setwise, then $G$ is a subgroup of $\Aut(\mathrm{s}(\Delta,\mathcal{C}))$. Moreover, if $G$ is arc-transitive on $\Delta$, then $G$ can swap the two cycles of $\mathcal{C}$ passing through each vertex. Hence, $G$ acts transitively on $V\mathrm{s}(\Delta,\mathcal{C})$. Observe, as well, that the $G$-invariance of $\mathcal{C}$ forces the action of any vertex-stabilizer on the neighbourhood of the fixed vertex to be imprimitive. We say that $G$ is \emph{locally-imprimitive}, for short.
    
    For any vertex $(\alpha,\mathbf{c}) \in V\mathrm{s}(\Delta,\mathcal{C})$, the vertex-stabilizer of $(\alpha,\mathbf{c})$ in $G$ is the intersection of the vertex-stabilizer of $\alpha \in V\Delta$ and the setwise stabilizer of the cycle $\mathbf{c}$. In particular, whenever $G$ is arc-transitive on $\Delta$, the order of the vertex-stabilizers in $\mathrm{s}(\Delta,\mathcal{C})$ is half that of those in $\Delta$. Furthermore, by a connectedness argument, these groups are both $2$-groups (see, for instance,~\cite[Lemma~3.1]{Spiga2014}).

    We have collected all the data about \cref{con:splitting} we need to proceed.
    
   \begin{lemma}\label[lemma]{lemma:girthSplit}
            Let $\Delta$ be a connected $4$-valent graph, and let $G$ be an arc-transitive locally-imprimitive group of automorphisms of $\Delta$ (or, equivalently, every vertex-stabilizer in $G$ is a $2$-group). Suppose that $\mathcal{C}$ is a $G$-invariant $2$-factor decomposition that contains girth cycles. Then
            \[ \mathrm{girth} (\Delta) = \mathrm{girth} \left( \mathrm{s}(\Delta, \mathcal{C})\right) \,. \]
        \end{lemma}
    \begin{proof}
    Consider the bijection
    \[ \theta: \left\{(\alpha, \mathbf{c}) \in V\Delta \times \mathcal{C} \mid \alpha \in V\mathbf{c} \right\} \to V\mathrm{s}(\Delta, \mathcal{C}) \]
    that \cref{con:splitting} induces. Note that, for every $\mathbf{c} \in \mathcal{C}$, $\theta(V\mathbf{c},\mathbf{c})$ is a cycle in $\mathrm{s}(\Delta,\mathcal{C})$ with the same length as $\mathbf{c}$. Hence, we immediately obtain the inequality
    \[ \mathrm{girth} (\Delta) \ge \mathrm{girth} \left( \mathrm{s}(\Delta, \mathcal{C})\right) \,. \]
    
    Aiming for a contradiction, we suppose that
    \[ \mathrm{girth} (\Delta) > \mathrm{girth} \left( \mathrm{s}(\Delta, \mathcal{C})\right) \,. \]
    Let $\mathbf{d}$ be a girth cycle in $\mathrm{s}(\Delta, \mathcal{C})$. Since all the cycles in $\mathcal{C}$ are longer than $\mathbf{d}$, it follows that the vertices
    \[ V\theta^{-1}(\mathbf{d}) := \{ \alpha \in V\Delta \mid \theta(\alpha,\mathbf{c}) \in V\mathbf{d}, \hbox{ for some }\mathbf{c}\in \mathcal{C} \}\]
    form a cycle of $\Delta$ whose length is strictly less than $\mathrm{girth}(\Delta)$, a contradiction. Hence, the equality is attained, and the proof is complete.
    \end{proof}

As a consequence of \cref{lemma:girthSplit}, we can focus our attention on locally-imprimitive $4$-valent graphs.
    
\section{The \texorpdfstring{$4$}{4}-valent scenario}\label{sec:BassSerre}
In this section, our aim is to build, for every integer $\ell \ge 3$, a $4$-valent graph whose automorphism group is arc-transitive and locally-imprimitive, and whose girth is $\ell$. We will make use of Bass-Serre theory: we refer to~\cite{DicksDunwoodyGroups,SerreTrees} for terminology and background.

Let $G$ be the amalgamated product $D_4 *_{C_2} (C_2\times C_2)$, where
\[ D_4 = \langle a,b \mid a^4, b^2, \, a^b=a^{-1} \rangle \quad \hbox{and} \quad  C_2 \times C_2 = \langle ab \rangle \times \langle z \rangle \, \]
and the amalgamation is over $ \langle ab \rangle$. Recall that every $g\in G$ can be uniquely written in the \emph{reduced form}
\[ g = (ab)^\varepsilon a^{e_1}za^{e_2} \ldots a^{e_{i-1}}za^{e_i} \,,\]
where $\varepsilon\in \{0,1\}$, $i \in \mathbb{N}$, $e_1,  e_i\in \{0,1,2,3\}$, and $e_2, \dots , e_{i-1}\in \{1,2,3\}$
(see, for instance,~\cite[Theorem~I.1]{SerreTrees}).
Following~\cite{SerreTrees}, we define the \emph{length of $g$} to be the number of symbols $a$, $a^2$, $a^3$ and $z$ we use to write it in its reduced form. Observe that this length defines a pseudometric on $G$ in the usual fashion, and hence we can talk about balls of centre $1$,
\[ \B_\ast(k) = \left\{ g \in G \mid \len_\ast(g)\le k \right\} \,.\]

We now state a result that we will prove in \cref{sec:finiteIndex}.
\begin{lemma}\label[lemma]{lemma:finiteIndexNormal}
    Let $G$ be the amalgamated product $D_4 *_{C_2} (C_2\times C_2)$, where
    \[ D_4 = \langle a,b \mid a^4, b^2, \, a^b=a^{-1} \rangle \quad \hbox{and} \quad 
    C_2 \times C_2 = \langle ab \rangle \times \langle z \rangle \, \]
    and the amalgamation is over $ \langle ab \rangle$. Then, for every integer $\ell\ge 3$, there exists a finite index normal subgroup $N_\ell$ of $G$ such that
    \begin{enumerate}[$(a)$]
        \item $N_\ell \cap \B_\ast(2\ell -1) = \left\{1\right\}$,
        \item $N_\ell \cap \B_\ast(2\ell + 1) = \left\{1, \, (za^2)^{\pm \ell}, \, (a^{-1}(za^2)^{\pm \ell} a)^{\pm 1} \right\}$.
    \end{enumerate}
\end{lemma}

We now start the construction of our putative $4$-valent graph $\Delta$. We consider the graph of groups
    \[\begin{tikzcd}
        D_4 &&& C_2 \times C_2
        \arrow["C_2", no head, from=1-1, to=1-4]
    \end{tikzcd}\]
and we denote by $\mathcal{T}$ the standard tree obtained from this graph of groups. Recall that $G$ is a group of automorphisms of $\mathcal{T}$, and that, as $|D_4:C_2|=4$ and $|C_2\times C_2:C_2|=2$, the vertices labelled by right cosets of $D_4$ have valency $4$, while those labelled by right cosets of $C_2\times C_2$ have valency $2$.

Observe that there is a one-to-one correspondence between non-backtracking paths of even length starting at $D_4$ and right cosets of $D_4$. Moreover, the minimal path in $\mathcal T$ from $D_4$ to $D_4g$ has lenght $2k$, where $k$ is the number of occurrences of $z$ in the reduced form of $g$.

Note that $\mathcal{T}$ is the barycentric subdivision of the infinite $4$-valent tree $\mathcal{S}$ (that is, by exchanging each vertex labelled by right cosets of $C_2\times C_2$ with an edge joining its two neighbours, we obtain the infinite $4$-valent tree). Therefore, $G$ acts arc-transitively on $\mathcal{S}$.
The metrics of these graphs are such that two vertices (labeled by the right cosets of $D_4$) at distance $2h$ in $\mathcal{T}$ are at distance $h$ in $\mathcal{S}$.

Before proceeding, it is convenient to have a concrete example of a path in $\mathcal{S}$ spelled out. For instance, the word $za^{e_1}za^{e_2}$ defines the path
\[\begin{tikzcd}
    D_4 && D_4 za^{e_2} && D_4 za^{e_1}za^{e_2}
    \arrow[no head, from=1-1, to=1-3]
	\arrow[no head, from=1-3, to=1-5]
\end{tikzcd}\]
while its inverse $a^{-e_2}za^{-e_1}z$ corresponds to the picture
\[\begin{tikzcd}
    D_4 && D_4 z && D_4 za^{-e_1}z
    \arrow[no head, from=1-1, to=1-3]
	\arrow[no head, from=1-3, to=1-5]
\end{tikzcd}\]

Let $N_\ell$ be the finite index normal subgroup of $G$ described in \cref{lemma:finiteIndexNormal}, and let $\Delta_\ell = \mathcal{S}/N_\ell$ be the finite $4$-valent graph obtained by quotienting $\mathcal{S}$ by $N_\ell$ (that is, the graph whose vertex-set and edge-set are the $N_\ell$-orbits of $V\mathcal{S}$ and $E\mathcal{S}$, respectively). We want to show that $\Delta_\ell$ satisfies the following additional properties.

\begin{lemma}\label[lemma]{lemma:BassSerre}
    For every positive integer $\ell\ge 3$, there exists a finite connected $4$-valent graph $\Delta_\ell$ such that
    \begin{enumerate}[$(a)$]
        \item $\Aut(\Delta_\ell)$ is arc-transitive and locally-imprimitive,
        \item each vertex-stabilizer has order at least $8$,
        \item the $\mathrm{Aut}(\Delta_\ell)$-invariant $2$-factor decomposition consists of girth cycles,
        \item $\mathrm{girth}(\Delta_\ell)= \ell$.
    \end{enumerate}
\end{lemma}
\begin{proof}
Let $\alpha$ be the vertex of $\Delta_\ell$ where the vertex $D_4$ of $\mathcal{T}$ projects to through the operations we have previously described. First, we observe that $G/N_\ell$ is isomorphic to a subgroup of $\mathrm{Aut}(\Delta_\ell)$. Observe that $D_4 \subseteq \B_\ast(1)$, thus \cref{lemma:finiteIndexNormal}~$(a)$ implies that $D_4 \cap N_\ell$ is trivial. It follows that the vertex-stabilizer of $\alpha$ contains
\[G_{D_4}N_\ell/N_\ell = D_4N_\ell/N_\ell \cong D_4/(D_4 \cap N_\ell) = D_4 \,,\]
which is a subgroup of order $8$. This proves $(b)$.
    
Let us denote by $\alpha_z$, $\alpha_{za}$, $\alpha_{za^2}$ and $\alpha_{za^3}$ the neighbours of $\alpha$, which can be recognized as the projection of the vertices $D_4 z$, $D_4 za$, $D_4 za^2$ and $D_4 za^3$, respectively. Observe that the action of $D_4$ on the neighbourhood of $\alpha$ defines the blocks of imprimitivity $\{\alpha_z, \alpha_{za^2}\}$ and $\{\alpha_{za}, \alpha_{za^3}\}$. Furthermore, using \cref{lemma:finiteIndexNormal}~$(b)$ and recalling the correspondence between non-backtracking paths and words, the girth cycles through $\alpha$ in $\Delta_\ell$ have length $\ell$ and they contain either the pair $\{\alpha_z, \alpha_{za^2}\}$ or $\{\alpha_{za}, \alpha_{za^3}\}$ (the former being induced by the word $(za^2)^{\pm \ell}$ and the latter by $(a^{-1}(za^2)^{\pm \ell} a)^{\pm 1}$). Therefore, $(d)$ holds.

Moreover, \cref{lemma:finiteIndexNormal}~$(b)$ also implies that there exists no girth cycle that contains any pair of neighbours of $\alpha$ distinct from $\{\alpha_z, \alpha_{za^2}\}$ and $\{\alpha_{za}, \alpha_{za^3}\}$. It follows that $\mathrm{Aut}(\Delta_\ell)_\alpha$ must define the same two blocks of imprimitivity on the neighbourhood of $\alpha$ as $D_4$ does, and hence $(a)$ holds. Last, we can note that the $\mathrm{Aut}(\Delta_\ell)$-invariant $2$-factor decomposition contains precisely the $\mathrm{Aut}(\Delta_\ell)$-orbits of these two girth cycles through $\alpha$. Thus, $(c)$ is proved.
\end{proof}

We are now fully equipped to prove the main theorem of this note.

\begin{proof}[Proof of \cref{thm}]
    For every $\ell\ge 3$, let $\Delta_\ell$ be the $4$-valent graph described in \cref{lemma:BassSerre}, and let $\mathcal{C}$ be the $\Aut(\Delta_\ell)$-invariant $2$-factor decomposition of $\Delta_\ell$. Using \cref{lemma:girthSplit} and the fact that the order of every vertex-stabilizer is halved by applying \cref{con:splitting}, we obtain that
    \[ \Gamma_\ell = \mathrm{s}(\Delta_\ell, \mathcal{C})\]
    enjoys all the properties required by \cref{thm}.
\end{proof}

\section{Subgroups consisting of long words}\label{sec:finiteIndex}
In this section, we prove \cref{lemma:finiteIndexNormal}. We recall some notation about word lengths.

In a finitely generated group \(G\), every finite generating set \(S\) gives rise to a word metric on the Cayley graph. We will write \(\len_S(g)\) for the distance from \(1\) to \(g\), and \(\B_S(\ell)\) for the ball of radius \(\ell\) around the identity (when \(G\) is a free group, we will usually drop the subscript $S$, as we refer to the canonical generators). Note that this notation is consistent with that introduced in Section~\ref{sec:BassSerre}, because, in that context, it corresponds to the length in the coset space $(D_4 *_{C_2} C_2 \times C_2) / C_2$ with respect to the generating set $\{a, a^2, a^3, z \}$.

\begin{lemma}
\label[lemma]{lem: res fin}
Let $G$ be a finitely generated group with finite generating set $S$, and let $H \unlhd G$ be a normal subgroup. Suppose that $G/H$ is residually finite. Then, for every positive integer $r$, there exists a finite index normal subgroup \( K_r \trianglelefteq G \) containing \( H \) such that 
\[ K_r \cap \B_S(r) = H \cap \B_S(r).
\]
\end{lemma}
\begin{proof}
Since $G/H$ is residually finite, there exists a family of finite index normal subgroups $\{ H_i \}_{i \in I}$ such that 
\[ H = \bigcap_{i \in I} H_i.\]
For each $w \in \B_S(r) - H$, there exists $i(w) \in I$ such that $w$ is not an element of $H_{i(w)}$. Accordingly, consider the finite index normal subgroup 
\[ K_r = \bigcap_{w \in \B_{S}(r) - H} H_{i(w)}.\]
Then $H$ is a subgroup of $K_r$, and $K_r$ is disjoint from $\B_S(r) - H$. Therefore, $K_r$ is the required normal subgroup.
\end{proof}

We present a preliminary computational lemma.

\begin{lemma}
    \label[lemma]{lem:explicitQuot}
    Let $G$ be the amalgamated product $D_4 *_{C_2} (C_2\times C_2)$, where
    \[ D_4 = \langle a,b \mid a^4, b^2, \, a^b=a^{-1} \rangle \quad \hbox{and} \quad 
    C_2 \times C_2 = \langle ab \rangle \times \langle z \rangle \, \]
    and the amalgamation is over $ \langle ab \rangle$, and, for an integer $\ell \ge 3$, let $H_\ell$ be the normal closure $\langle (za^2)^\ell \rangle^G$. Then,
    \[ G/H_\ell \cong D_4 *_{C_2\times C_2} (C_2 \times D_\ell) \,,\]
    where
    \[ C_2 \times C_2 = \langle ab, a^2 \rangle, \quad \hbox{and}\quad
    C_2 \times D_\ell = \langle ab, a^2, z \rangle \,.\]
\end{lemma}

\begin{proof}
We introduce the groups $A=\langle a,b\rangle\cong D_4$ and
$E=\langle ab,a^2\rangle\cong C_2\times C_2$. On the one hand, note that
    \[ [ab, a^2] = b^{-1}a^{-1}a^2aba^2 = (a^2)^b a^2 = 1 \,\]
    and that
    \[ (ab)^2=(abab) =a a^{-1} =1 \, .\]
    Thus, we can rewrite the presentation of $G$ as
    \[ G = \langle a,b, z \mid a^4, \, b^2, \, z^2, \, baba, \, [ab,z] \rangle \,. \]
    The presentation of the quotient is obtained by imposing the additional relation $(za^2)^\ell=1$, that is,
    \[ G/H_\ell = \langle a, b, z \mid a^4, \, b^2, \, z^2, \, baba, \, [ab,z],\, (za^2)^\ell \rangle \,.\]

On the other hand, we consider the group
\[
B_\ell
= \langle c\mid c^2\rangle
  \times \langle d,z\mid d^2,z^2,(zd)^\ell\rangle
\cong C_2\times D_\ell.
\]

The assignments $ab\mapsto c$ and $a^2\mapsto d$ identify $E$ with the subgroup $\langle c,d\rangle$ of $B_\ell$. We form $A*_E B_\ell$ using this identification.
Since $(ab)^2=(a^2)^2=[ab,a^2]=1$, and eliminating $c$ and $d$, we write the presentation of
$A*_E B_\ell$ as
\[ A*_E B_\ell \cong \langle a,b,z\mid a^4,b^2,z^2,\, baba, \, [ab,z], \, (za^2)^\ell \rangle \, .\]
This is precisely the presentation of $G/H_\ell$, as required.
\end{proof}

After this preliminary discussion, we can tackle \cref{lemma:finiteIndexNormal}.

\begin{proof}[Proof of \cref{lemma:finiteIndexNormal}]
We first analyse the geometry of the graph ${\mathcal S}_\ell=\mathcal S/H_\ell$. We keep the symbols for the subgroups of $G$ and $G/H_\ell$ used in the proof of \cref{lem:explicitQuot}, and we introduce
\[D = \langle ab, z \rangle \cong C_2 \times C_2 \,.\]

Recall that the vertices and edges of $\mathcal S$ are labelled by the right cosets of $A$ and of $D$ in $G$, respectively. Meanwhile, the vertices of ${\mathcal S}_\ell$ are the double cosets $AgH_\ell$,  while its edges are labelled by the double cosets $Dg H_\ell$, for some $g\in G$.
Consider the group element $\tau = a^2 z H_\ell \in B_\ell$. For a right coset $B_\ell g$, consider the cycle $\mathcal{C}_g$ with vertex- and edge-set
\begin{gather*}
    V\mathcal{C}_g = \left\{ A \tau^i g \mid i \in \mathbb{Z}\right\} \subseteq V\mathcal{S}_\ell,\\
    E\mathcal{C}_g = \left\{ D \tau^i g \mid i \in \mathbb{Z} \right\} \subseteq E\mathcal{S}_\ell.
\end{gather*}
We claim that every simple cycle of $\mathcal{S}_\ell$ is of the form $\mathcal{C}_g$, and has length $\ell$.

Since $a^2\in A$, $D\tau^i g$ is the edge with endpoints $A\tau^i g$ and $Az\tau^ig = A\tau^{i+1}g$. Hence, $\mathcal{C}_g$ is indeed a cycle.
The latter part of the claim is a straightforward computation. Since $A\cap B_\ell=E$ and
$|B_\ell:E|=|B_\ell:D|=\ell$, we obtain the correct length, that is, $|V\mathcal{C}_g|=|E\mathcal{C}_g|=\ell$. Moreover, we observe that $\{\mathcal{C}_g \mid g\in G\}$ is a $2$-factor decomposition of $\mathcal{S}_\ell$. Since $D\leq B_\ell$, every edge $DgH_\ell$ lies in $\mathcal C_{g}$, and, as $|A:A\cap B_\ell|=|A:E|=2$, every vertex $AgH_\ell$ lies in the two cycles $\mathcal C_g$ and $\mathcal C_{ag}$. 

We want to rule out any other simple cycle. We consider the standard tree $\mathcal{R}$ obtained from the graph of groups
\[\begin{tikzcd}
        A\cong D_4 &&& B_\ell\cong C_2 \times D_\ell
        \arrow["E\cong C_2 \times C_2", no head, from=1-1, to=1-4]
    \end{tikzcd}\]
In view of the preceding remark, $\mathcal{R}$ is a bipartite graph which describes the incidence of vertices from $\mathcal{S}_\ell$ with cycles of the $2$-factor decomposition. Aiming for a contradiction, suppose that there is a simple cycle $\mathcal{D}$ in $\mathcal{S}_\ell$ that contains edges from more than one decomposition cycle $\mathcal C_{g}$. Colour $\mathcal D$ according to its maximal paths $P_1,\ldots,P_m$ contained in individual cycles of the $2$-factor decomposition. Let $ C_j$ be the cycles of the form $\mathcal{C}_g$ containing $P_j$, and let $v_{j-1}$ and $v_j$ be the endpoints of $P_j$ (with indices read cyclically). Then $v_0,  C_1, v_1,\ldots, C_m,v_0$ is a closed walk in $\mathcal R$. Note that, as $\mathcal{D}$ is simple, $v_{j-1}$ and $v_j$ are always distinct, and so are $C_j$ and $C_{j+1}$ by maximality of $P_j$. We find a contradiction because a tree has no closed walk without an immediate reversal. Therefore, all the simple cycles of $\mathcal{S}_\ell$ are of the form ${\mathcal C}_{g}$, and the claim is proved.

By \cref{lem:explicitQuot}, $G/H_\ell$ is an amalgamated product of finite groups, and hence residually finite. We can apply \cref{lem: res fin} with radius $2\ell+3$ and generating set $S=(A\setminus\{1\})\cup\{z\}$, and we obtain a finite-index normal subgroup $N_\ell\unlhd G$ containing $H_\ell$ such that
\[ N_\ell\cap\B_{S}(2\ell+3) = H_\ell\cap\B_{S}(2\ell+3) \,.\] 
In particular, as $\len_{S}(w)\leq\len_\ast(w)+1$,
\[ N_\ell\cap\B_\ast(2\ell+1)=H_\ell\cap\B_\ast(2\ell+1)\,. \]

To conclude the proof, it remains to compute the short words in $H_\ell$. Note that, by \cref{lem:explicitQuot}, we have $H_\ell\cap A=H_\ell\cap D=\{1\}$, and that $H_\ell$ is normal in $G$. Therefore, the projection $\mathcal S\to\mathcal S_\ell$ is a covering and so it preserves non-backtracking paths. Let $h\in H_\ell$ be a nontrivial element with $\len_\ast(h) \leq 2\ell+1$, and let $k$ be the number of occurrences of $z$ in the reduced form of $h$. Consider the geodesic path of length $k$ from $A$ to $Ah$ in the tree $\mathcal S$. Since $h\in H_\ell$, the endpoints have the same image in $\mathcal S_\ell$. The projected path is therefore a closed non-backtracking walk of length $k$. Hence, as every simple cycle of $\mathcal S_\ell$ has length $\ell$, we have $k\ge\ell$. Moreover, $\len_\ast(h)\ge2k-1$ implies $k\le\ell+1$.
We can exclude $k=\ell+1$. If some vertex appeared twice before the final return, the subwalk between those two visits would contain a simple cycle and hence have at least $\ell$ edges. The complementary subwalk would also be closed, but would have at most one edge, which is impossible because $\mathcal S_\ell$ has no loops. Thus, the walk would be a simple cycle of length $\ell+1$, contrary to the preceding claim. 

Hence, $k=\ell$. The geodesic path from $A$ to $Ah$ projects onto one of the two $\ell$-cycles through $AH_\ell$ in one of its orientations. These cycles correspond to $\mathcal{C}_1$ and $\mathcal{C}_a$, respectively. Since $H_\ell \cap A = 1$, each oriented cycle gives only one element of $H_\ell$. Explicitly,
\begin{equation}
    \label{eq:algo}
 H_\ell\cap\B_\ast(2\ell-1)=\{1\}, \quad \hbox{and} \quad 
 H_\ell\cap\B_\ast(2\ell+1)
 =\{1,(za^2)^{\pm \ell},(a^{-1}(za^2)^\ell a)^{\pm1}\} \,.
\end{equation}
This completes the proof of \cref{lemma:finiteIndexNormal}.
\end{proof}

Observe that the previous proof also shows that $\mathcal S \to\mathcal S/N_\ell$ is a graph covering factoring through $\mathcal S_\ell$. Furthermore, no two vertices of a simple $\ell$-cycle are identified, and every $\ell$-cycle in the quotient must lift to a cycle in $\mathcal S_\ell$ from the $2$-factor decomposition. This consideration gives a straightforward proof of the imprimitivity of the local action that was provided in \cref{lemma:BassSerre}.

\section{Effective bounds for the number of vertices}
\label{sec: effective}
To obtain an effective bound on the number of vertices of the graphs $\Gamma_\ell$, it is necessary to control the index of the normal subgroup of prescribed girth appearing in \Cref{lemma:finiteIndexNormal}. To estimate this value, we study the girth of certain verbal subgroups in a free group. Specifically, we consider $F_k^\ell \cdot \gamma_\ell(F_k)$, the product of the $\ell$-power subgroup
\[ F_k^\ell = \langle g^\ell \mid g \in F_k \rangle,\]
and the $\ell$-th term of the lower central series, defined iteratively as $\gamma_1(G) = G$ and
\[ \gamma_i(G) = [\gamma_{i-1}(G), \, G] \, \textup{ for all } i \geq 2.\]

Note that $F_k^\ell \cdot \gamma_\ell(F_k)$ has finite index in $F_k$, by virtue of the solution of the genuine Burnside Problem for nilpotent groups due to Hirsch (see, for instance,~\cite[Lemma~2.34]{CeSi}).

\begin{proof}[Proof of \cref{cor: p-power-nilpotent}]
First, in order to prove the inequality, we aim to prove that there exists a $k$-generated nilpotent group $G$ of exponent $p$ and nilpotency class at most $p-1$, where no nontrivial word of length less than $p$ vanishes on its chosen generators. Then $G=F_k/N$ for a normal subgroup $N\unlhd F_k$ such that $F_k^p\cdot\gamma_p(F_k)\leq N$ and $N$ contains no nonidentity word of length less than $p$.

Let $w=x_{i_1}^{\alpha_1}\dots x_{i_n}^{\alpha_n}$ be a nontrivial reduced word, $\alpha_j\in\Z$ and $i_j\in\{1,\dots,k\}$ with distinct adjacent indices. Assume that $\len(w)<p$, so that $|\alpha_j|<p$ for all $j$. Consider the Magnus embedding
\[
\mathcal M\colon F_k\to\mathbb Z\langle\!\langle X_1,\dots,X_k\rangle\!\rangle,
\qquad x_i\mapsto1+X_i,
\]
into the $\mathbb Z$-algebra of formal power series in $k$ non-commuting variables, see~\cite[Proposition~10.1]{LyndonSchuppCombinatorial}. The coefficient $c_W\in\Z$ of the monomial $W=X_{i_1}\dots X_{i_n}$ in $\M(w)$ is given by
\[\varepsilon(\partial_{i_1}\dots\partial_{i_n}(w)),\]
where $\partial_j$ is the Fox derivative $\partial_j\colon\Z[F_k]\to\Z[F_k]$ defined as $\partial_j(x_i)=\delta_{ij}$ and $\varepsilon\colon\Z[F_k]\to\Z$ is the augmentation map, see~\cite[Chapter~I.10]{LyndonSchuppCombinatorial}. Moreover,~\cite[Equation~(4)]{MaPu} implies
\[c_W=\varepsilon(\partial_{i_1}\dots\partial_{i_n}(w))=\prod_{j=1}^n\alpha_j.\]

Let $\F_p$ be the finite field of $p$ elements, let $\M_p\colon F_k\to\F_p\langle\!\langle X_1,\dots,X_k\rangle\!\rangle$ be the Magnus embedding where the coefficients are reduced modulo $p$, and let
\[\mathcal I_p=\left\{f\in\F_p\langle\!\langle X_1,\dots,X_k\rangle\!\rangle\mid\mathrm{ord}(f)\geq p\right\}.\footnote{Here $\mathrm{ord}(f)$ is the least total degree of any monomial that occurs in $f$ with a nonzero coefficient. By convention the zero series has infinite order.}\]
Note that $\M_p(F_k)\cap(1+\mathcal I_p)$ is a normal subgroup of $\M_p(F_k)$, because conjugation preserves the order filtration. Moreover, the total degree of $W$ satisfies
\[\deg(W)=n\leq\len(w)<p,\]
and, since $|\alpha_j|<p$ for every $j\in\{1,\dots,n\}$, $c_W$ and $p$ are coprime. Therefore $\M_p(w)\notin1+\mathcal I_p$.

We now define the group $G=\M_p(F_k)/(\M_p(F_k)\cap(1+\mathcal I_p))$.
To complete the proof, we will show that $G$ is nilpotent of class at most $p-1$ and has exponent $p$. For the former, we observe that $\M_p(w)\in1+\mathcal I_p$ whenever $w\in\gamma_p(F_k)$ (see~\cite[Proposition~10.2]{LyndonSchuppCombinatorial}). For the latter, let $g\in F_k$, then
\[\M_p(g^p)=\M_p(g)^p=(1+f)^p=1+f^p,\]
for some power series $f$ such that $f(0,\dots,0)=0$. Write $f$ as the sum of monomials $\sum f_i$ where $\deg(f_i)\geq1$, then
\[\deg(f_{i_1}\dots f_{i_p})\geq\sum_{j=1}^p\deg(f_{i_j})\geq p.\]
Thus $f^p\in\mathcal I_p$. This proves that $G$ has exponent $p$.

We now characterize equality. Suppose that $w\in F_k^p\gamma_p(F_k)$ satisfies $\len(w)=p$. If $p$ is odd, then $w\notin\gamma_2(F_k)=[F_k,F_k]$. If $p=2$, there exists no freely reduced word of length $2$ in $\gamma_2(F_k)$, so again $w\notin\gamma_2(F_k)$. Consider the abelianisation map $\pi\colon F_k\to\Z^k=F_k/\gamma_2(F_k)$. Since $w\in F_k^p\gamma_p(F_k)-\gamma_2(F_k)$, we have $w^\pi\in p\Z^k-\{0\}$, and hence
\[
    p=\len(w)\geq\len_{\{x_1^\pi,\dots,x_k^\pi\}}(w^\pi)\geq p.
\]
That is, $w^\pi=\pm p\,x_i^\pi$ for some $i\in\{1,\dots,k\}$, and hence $w=x_i^{\pm p}$.
\end{proof}

\begin{remark}
    \label[remark]{remark:sonoUnaBellissimaFarfalla}
    The conclusion of \cref{cor: p-power-nilpotent} cannot extend to arbitrary positive integers. Indeed, consider the nontrivial reduced word $w(x_1, x_2)= \left[x_1^4,x_2^5 \right]$,  which has length $18$. For $k\geq 2$, $G=F_k/F_k^{20}\gamma_{20}(F_k)$ is a finite nilpotent group of exponent dividing $20$. Since the images of $x_1^4$ and $x_2^5$ in $G$ have orders dividing $5$ and $4$, respectively, they belong to distinct Sylow subgroups, and hence they commute. Therefore, $w(x_1,x_2)\in F_k^{20}\gamma_{20}(F_k)$, despite $w$ having length strictly less than $20$. (Nevertheless, the corresponding inequality does hold for nontrivial words belonging to $\gamma_\ell(F_k)$ also for a composite integer $\ell$ by \cite[Theorem 1.2]{MaPu}.)
\end{remark}

\begin{proof}[Proof of \cref{prop:bounds}]
Let $\mathcal R$ be the standard tree corresponding to $G/H_\ell = A*_{E}B_\ell$, as introduced in the proof of \cref{lemma:finiteIndexNormal}. Since $|A|=8$ and $|B_\ell| = 4\ell$, we can let each group act semiregularly on a set of size $8\ell$. These actions can be restricted to $E$, and, after eventually relabelling, we can assume that they do agree. Hence, the universal property of amalgamated product gives a homomorphism $\rho \colon G/H_\ell\to \mathrm{Sym}(8\ell)$ that is injective on both factors. Its kernel $\ker(\rho)$ is normal and free by~\cite[Proposition~I.7.10]{DicksDunwoodyGroups}. Observe that
\[ | G/H_\ell : \ker(\rho) | \le (8\ell)! \,,\]
and, by \cite[Remark~IV.1.11]{DicksDunwoodyGroups}, the rank of $\ker(\rho)$ is
\[ 1 + | G/H_\ell : \ker(\rho) | \left(\frac1{|E|}-\frac1{|A|}-\frac1{|B_\ell|}\right) \le 1 + \frac{(8\ell)! \cdot (\ell-2)}{8\ell}\le (8\ell)! \,. \]

Let $\pi \colon G\to G/H_\ell$ be the quotient map. Our aim is to determine a finite-index normal subgroup $M \lhd G/H_\ell$ contained in $\ker(\rho)$ such that
\[ M \cap \pi\bigl(\B_\ast(2\ell+1)\bigr) = \{1\} \,. \]
If such $M$ exists, then its preimage under $\pi$ would then satisfy
\[ \pi^{-1}(M) \cap \B_\ast(2\ell+1) = H_\ell \cap \B_\ast(2\ell+1) \,.\]
Therefore, by \cref{eq:algo}, $\pi^{-1}(M)$ also satisfies the conditions of \cref{lemma:finiteIndexNormal}, thus completing this proof.

We consider a spanning tree of $\mathcal R/ \ker(\rho)$, and we choose an orientation for each edge outside it. These edges give a free basis $S$ of
\[ \ker(\rho) \cong \pi_1(\mathcal R/\ker(\rho)) \,.\]
Every closed path in $\mathcal R/\ker(\rho)$ corresponds to a word $u \in \ker(\rho)$ by recording the generators of $S$ traced out by the path. Moreover, the geodesic segment in $\mathcal R$ from $A$ to $Au$ projects to a closed path. Hence,
\[ \len_S(u) \leq d_{\mathcal R}(A, Au ) \,.\]

Let $w$ be a word in $\B_*(2\ell+1)$. Suppose that
$\pi(w)\in\ker(\rho)$ is nontrivial.
The quotient homomorphism $\pi:G\to G/H_\ell$ induces
a simplicial map $\mathcal T\to\mathcal R$, where  $\mathcal T$ is the Bass-Serre tree associated with $G$. This map sends 
each vertex $Ag$ to $A\pi(g)$ and each vertex $Dg$
to $B_\ell\pi(g)$. Since this map preserves adjacency,
it does not increase distances. Consequently,
\[
\len_S(\pi(w))
\le d_{\mathcal R}(A,A\pi(w))
\le d_{\mathcal T}(A,Aw)
\le \len_*(w)+1
\le 2\ell+2 \, .
\]

Now we choose a prime $p>2\ell+2$, which implies that
\[ p > \max \left\{ \len_S(u) \mid u \in \ker(\rho) \cap \pi (\B_\ast(2\ell+1)) \right\} \,.\]
We define $M_p = (\ker\rho)^p\gamma_p(\ker\rho)$. As $M_p$ is verbal, this subgroup is characteristic in $\ker\rho$, hence normal in $G/H_\ell$. Moreover, by construction,
\[ | \ker(\rho) : M_p| = \mathsf{N}(\operatorname{rk}(\ker(\rho)),p,p) \,. \]
By \cref{cor: p-power-nilpotent}, every nontrivial element of $M_p$ has length at least $p$ with respect to $S$. Therefore,
\[ M_p \cap \pi\bigl(\B_\ast(2\ell+1)\bigr) = \{1\} \,, \]
as desired.

Since $\pi^{-1}(M_p)$ satisfies the conditions of \cref{lemma:finiteIndexNormal}, we repeat the proof of \cref{thm} using $\pi^{-1}(M_p)$ instead of $N_\ell$, thus obtaining a suitable graph $\Gamma_\ell^p$. By construction,
\[\begin{split}
    |V\Gamma_\ell^p| &= \frac{1}{4}|G:\pi^{-1}(M_p)| \\
    &= \frac{1}{4}|G/H_\ell:M_p| \\
    &= \frac{1}{4}|G/H_\ell:\ker(\rho)|\cdot|\ker(\rho):M_p| \\
    &\le \frac{(8\ell)!}{4} \mathsf{N}(\operatorname{rk}(\ker(\rho)),p,p) \\
    &\le \frac{(8\ell)!}{4} \mathsf{N}((8\ell)!, \, p, \, p) \,.
\end{split}\]
This completes the proof of \cref{prop:bounds}.
\end{proof}

\begin{remark}\label[remark]{rem:infinitude}
    Observe that, for every integer $\ell\ge 3$, the proof of \cref{prop:bounds} provides a method to construct infinitely many finite connected flexible $3$-valent graphs $\Gamma_\ell^p$ of girth $\ell$ by varying the prime $p$. Indeed,
\[
    |V\Gamma_\ell^p|
    = \frac{|G/H_\ell:\ker(\rho)|}{4}\,
      \mathsf{N}(\textrm{rk}\ker(\rho), \, p, \, p) \, 
\]
where the first factor is independent of $p$ and the second is
a nontrivial power of $p$. Distinct primes therefore give graphs
of distinct orders, which are therefore non-isomorphic. 
\end{remark}

\bibliographystyle{plain}
\bibliography{bibFlexible}

\end{document}